\documentclass[12pt]{article}
\usepackage{amssymb,amsfonts,amsmath,amsthm,comment,dsfont,xcolor,tikz,tikz-cd}
\usepackage[colorlinks, citecolor=purple,linkcolor=blue,urlcolor=green]{hyperref}
\usepackage[lite,alphabetic]{amsrefs}
\usepackage[a4paper, portrait, margin=25mm]{geometry}

\newcommand{\1}{\mathds{1}}
\newcommand{\0}{\mathds{O}}

\newcommand{\R}{\mathbb{R}}

\newcommand{\N}{\mathbb{N}}

\newcommand{\Bo}{\mathcal{B}}

\newcommand{\8}{\infty}

\newcommand{\fvl}{\mathrm{FVL}}
\newcommand{\fbl}{\mathrm{FBL}}
\newcommand{\So}{\mathrm{S}}

\newcommand{\card}{\mathrm{card}}

\newcommand{\vlt}{\mathrm{vlt}}

\newcommand{\Co}{\mathcal{C}}

\newcommand{\Po}{\mathcal{P}}

\newcommand{\Int}{\mathrm{int}}

\newtheoremstyle{mytheorem}
  {12pt}   
  {4pt}   
  {\itshape} 
  {}       
  {\bfseries} 
  {.}      
  {.5em}   
  {}       

\newtheoremstyle{mydefinition}
  {12pt}        
  {8pt}       
  {\normalfont} 
  {}           
  {\bfseries}   
  {.}           
  {.5em}        
  {}

\newcounter{dummy} \numberwithin{dummy}{section}
\theoremstyle{mytheorem}
\newtheorem{theorem}[dummy]{Theorem}
\newtheorem{lemma}[dummy]{Lemma}

\newtheorem{proposition}[dummy]{Proposition}
\newtheorem{corollary}[dummy]{Corollary}
\newtheorem{question}[dummy]{Question}
\theoremstyle{mydefinition}
\newtheorem{remark}[dummy]{Remark}
\newtheorem{example}[dummy]{Example}
\usepackage{scalerel}

\begin{document}

\title{The Dedekind completion of an Archimedean ordered vector space as a reflector}

\author{Antonio Avil\'{e}s\thanks{Departamento de Matem\'{a}ticas, Universidad de Murcia, Spain (\texttt{avileslo@um.es})}~ and Eugene Bilokopytov\thanks{Instituto de Ciencias Matem\'aticas, Madrid, Spain (\texttt{bilokopy@ualberta.ca}).}}

\maketitle

\begin{abstract}
We consider the category $\mathbf{AOVS}$ of Archimedean ordered vector spaces with linear maps which preserve all existing nonempty suprema, and its full subcategories $\mathbf{DAOVS}$, $\mathbf{DVL}$ and $\mathbf{UVL}$, consisting of directed spaces, Dedekind complete vector lattices and universally complete vector lattices, respectively. We deduce from some results in the literature that $\mathbf{DVL}$ and $\mathbf{UVL}$ are reflective subcategories of $\mathbf{DAOVS}$, with the usual Dedekind completion being the reflector in $\mathbf{DVL}$. In contrast to these facts, we show that a non-directed Archimedean ordered vector space of dimension greater than $1$ has no reflector in either $\mathbf{DVL}$ or $\mathbf{UVL}$. In particular, there are no free Dedekind complete vector lattices over a set with more than one element. We also use the occasion to show that a free vector lattice with $\alpha$ generators embeds into a free vector lattice with $\beta$ generators if and only if $\alpha\le\beta$, and explore the concept of the free completion of an Archimedean vector lattice with a strong unit.\medskip


MSC2020 46A40, 46E05.
\end{abstract}

\section{Introduction}

The theory of ordered vector spaces and vector lattices is rich in completion procedures. Among these procedures, Dedekind completion occupies a central place. In the setting of vector lattices, this construction is classical and has been studied from many viewpoints. Nevertheless, when one moves from vector lattices to more general Archimedean ordered vector spaces, and especially when one adopts a categorical perspective, several natural questions become delicate.

A useful way to view completion phenomena is through reflectivity. Given a category of ordered structures together with a distinguished full subcategory of ``complete'' objects, one may ask whether every object admits a universal arrow into that subcategory. This approach to the MacNeille completion of posets was implemented in \cite{bishop} and \cite{erne} for various classes of morphisms. In the present context, the relevant ambient category is formed by Archimedean ordered vector spaces with linear maps preserving all existing nonempty suprema, while the principal subcategories are those of Dedekind complete vector lattices, and universally complete vector lattices. Here we extend the scope of \cite{al} and \cite{bmo}, which dealt only with Archimedean vector lattices.

The first aim of this paper is to characterize the supremum-continuous operators on directed Archimedean ordered vector spaces in terms of their behavior under completion. More precisely, we show in Theorem \ref{supcont1} that such an operator admits a unique extension between Dedekind completions, thereby proving that the category of Dedekind complete vector lattices is reflective in the category of directed Archimedean ordered vector spaces, with the usual Dedekind completion serving as the reflector. By combining this with a result of Fremlin about the universal completion of a Dedekind complete vector lattice, one also obtains a corresponding reflectivity statement for universally complete vector lattices. We note that Theorem \ref{supcont1} is closely related to some results in \cite{kvgs}.

The second aim is to show that this satisfactory picture breaks down sharply outside the directed setting. One might hope that a suitably generalized notion of completion still produces reflectors for arbitrary Archimedean ordered vector spaces. Theorem \ref{main} proves that this is not the case: apart from the one-dimensional trivial-order example $\mathbb{R}_0$, non-directed Archimedean ordered vector spaces do not admit reflectors in either the Dedekind complete or the universally complete categories. On the way to this result, and as special case, we also show in Proposition \ref{nofree} that there are no free Dedekind complete vector lattices, nor free universally complete vector lattices, over a set with more than one element. Our methods draw substantially from the Boolean algebra literature; they have also inspired a ``side application'': motivated by a result from \cite{day} we have determined in Proposition \ref{day3} which Archimedean vector lattices with strong units have a completion defined as a free object.

Finally, since a significant part of the paper deals with free constructions, we took the occasion to revisit the notion of a free vector lattice over a set. Here we prove a basic fact (Theorem \ref{mfre}) that a free vector lattice on $\alpha$ generators embeds into a free vector lattice on $\beta$ generators if and only if $\alpha \le \beta$.

\section{Extensions of operators on ordered vector spaces}\label{sectsc}

Let $P$ be a poset. We denote $Q^{\Uparrow}:=\left\{p\in P,~ p\ge Q\right\}$, i.e. the collection of all upper bounds of $Q\subset P$. If $R\subset P$, we further denote $Q^{\Uparrow_{R}}:=Q^{\Uparrow}\cap R$. The meaning of the notations $Q^{\Downarrow}$ and $Q^{\Downarrow_{R}}$ is similar. We say that $Q$ is \emph{supremum-dense} or \emph{$\bigvee$-dense} in $P$ if $p^{\Downarrow_{Q}}\ne\varnothing$ and $p=\bigvee p^{\Downarrow_{Q}}$, for every $p\in P$, and define \emph{infimum-denseness} analogously. We say that $Q$ is \emph{supremum-closed}, or \emph{$\bigvee$-closed} if whenever $\varnothing\ne R\subset Q$ and $q=\bigvee R$, it follows that $q\in Q$.

If $\varphi:P\to R$ is a map between posets $P$ and $R$, we call it \emph{supremum-continuous} or \emph{$\bigvee$-continuous}, if it preserves all existing nonempty suprema in $P$, i.e. if $\varnothing\ne Q\subset P$ and $q\in P$ are such that $q=\bigvee Q$, then $\varphi\left(q\right)=\bigvee\varphi\left(Q\right)$. If $p\le q$, taking $Q=\left\{p,q\right\}$ yields $\varphi\left(q\right)=\varphi\left(p\right)\vee \varphi\left(q\right)$, in other words $\varphi\left(p\right)\le \varphi\left(q\right)$. It follows that every $\bigvee$-continuous map is order preserving. Every order isomorphism is supremum-continuous. More generally, if $Q\subset P$ is supremum-dense and infimum-dense, then the embedding of $Q$ into $P$ is supremum-continuous (for the case of ordered vector spaces see \cite[Proposition 2.3.27]{kv}). A pre-image of a $\bigvee$-closed set under a $\bigvee$-continuous map is  $\bigvee$-closed.\medskip

Assume that $E$ is an ordered vector space.\footnote{Here "ordered" means "partially ordered".} Then a linear subspace $H\subset E$ is supremum-dense if and only if it is infimum-dense, and these conditions also imply that $H$ is \emph{majorizing}, i.e. for every $e\in E$ there is $h\in H$ with $h\ge e$. Recall that $E$ is a \emph{Dedekind complete vector lattice} if it is directed and every nonempty order bounded set has a supremum in $E$ (in particular, every set of two elements). It is easy to see that a $\bigvee$-closed subspace of a Dedekind complete vector lattice is also a Dedekind complete vector lattice. If $E$ is directed and Archimedean, then it has a \emph{Dedekind completion} $E^{\delta}$, which is the unique Dedekind complete vector lattice which contains $E$ as a supremum-dense subspace (see \cite[Section 2.1]{kv}).

Let $E,F$ be ordered vector spaces and let $L\left(E,F\right)$ be the space of all linear operators from $E$ into $F$. Recall that an operator $T\in L\left(E,F\right)$ is \emph{positive} if $TE_{+}\subset F_{+}$, which is equivalent to being order preserving. Moreover, 
$T\in L\left(E,F\right)$ is supremum-continuous if and only if it preserves the nonempty sets whose infimum is zero.\footnote{In the literature these operators are often called complete Riesz homomorphisms, however such a name would be inconsistent with the nomenclature that we are developing in this paper.} Clearly, this property is stable under compositions. Recall that supremum-continuous linear maps between vector lattices are precisely order continuous homomorphisms (see e.g. \cite[Theorem 3.10]{erz}).

\begin{example}\label{trivial}
If $E_{+}=\left\{0_{E}\right\}$, then the only nonempty set in $E$ whose infimum is $0_{E}$ is $\left\{0_{E}\right\}$. It follows that any linear operator from $E$ into any ordered vector space is $\bigvee$-continuous.
\qed\end{example}

\begin{example}\label{bigv}
Let $H:=E_{+}-E_{+}$. Let us show that the inclusion $I_{H}:H\to E$ is $\bigvee$-continuous. Let $\varnothing\ne P\subset H$ be such that $\bigwedge_{H} P=0_{E}$. Clearly, $P\subset H_{+}=E_{+}$. Assume that $e\in E$ is such that $e\le P$. Take any $p\in P\subset E_{+}$. Then, $e=p-\left(p-e\right)\in E_{+}-E_{+}=H$, and since $e\le P$ and $\bigwedge_{H} P=0_{E}$, it follows that $e\le 0_{E}$.\medskip

Let us also show that an operator $T:E\to F$ is $\bigvee$-continuous if and only if $\left.T\right|_{H}$ is $\bigvee$-continuous. Necessity follows from $\left.T\right|_{H}=T I_{H}$. For sufficiency assume that $\varnothing\ne P\subset E$ is such that $\bigwedge P=0_{E}$. Then, $P\subset E_{+}\subset H$ and $\bigwedge_{H}P=0_{E}$. Hence, due to $\bigvee$-continuity of $\left.T\right|_{H}$ we get $\bigwedge TP=0_{F}$.
\qed\end{example}

The results of this section are partially contained in \cite[Section 3]{kvgs} and \cite{al}, but our approach is somewhat different. We start with an auxiliary result. Note that part (ii) is \cite[Proposition 10]{kvgs} (essentially \cite[Lemma 4.3]{imhoff}), and part (i) is essentially \cite[Lemma 3.4]{stennder} but we provide the proofs for the reader's convenience.

\begin{proposition}\label{line}
Let $E,F$ be ordered vector spaces, let $H\subset E$ be a supremum-dense linear subspace, and let $T:E\to F$ be a map. Then:
\item[(i)] If $T\in L\left(E,F\right)_{+}$ and such that $\left.T\right|_{H}$ is an order embedding, then $T$ is an order embedding.
\item[(ii)] If $T\in L\left(E,F\right)_{+}$ and such that $\left.T\right|_{H}$ is supremum-continuous, then $T$ is supremum-continuous.
\item[(iii)] If $T$ is supremum-continuous and such that $\left.T\right|_{H}$ is linear, then $T$ is linear.\medskip

Moreover, in (ii) and (iii) $T$ is uniquely determined by $\left.T\right|_{H}$.
\end{proposition}
\begin{proof}
(i): Assume that $e\in E$ is such that $Te\ge 0_{F}$. Then, for every $h\in e^{\Uparrow_{H}}$ we have $Th\ge Te\ge 0_{F}$, and since $\left.T\right|_{H}$ is an order embedding it follows that $h\in H_{+}$. Hence, $e^{\Uparrow_{H}}\subset H_{+}$, and since $H$ is infimum dense we conclude that $e=\bigwedge e^{\Uparrow_{H}}\ge 0_{E}$.\medskip

(ii): Assume that $\varnothing\ne P\subset E$ is such that $\bigwedge P=0_{E}$. Clearly, $TP\subset F_{+}$. For every $p\in P$ we have $p=\bigwedge p^{\Uparrow_{H}}$, and so $\bigwedge \bigcup\limits_{p\in P}p^{\Uparrow_{H}}=\bigwedge P=0_{E}$. It then follows that $\bigwedge_{H} \bigcup\limits_{p\in P}p^{\Uparrow_{H}}=0_{E}$, and since $\left.T\right|_{H}$ is supremum-continuous, we get $\bigwedge\bigcup\limits_{p\in P}Tp^{\Uparrow_{H}}=0_{F}$. If $f\in F$ is such that $f\le TP$, then for every $p$ we have $f\le Tp\le Tp^{\Uparrow_{H}}$. Hence, $f\le \bigcup\limits_{p\in P}Tp^{\Uparrow_{H}}$, and so $f\le 0_{F}$. We conclude that $\bigwedge TP=0_{F}$.\medskip

(iii): First, for $e,g\in E$ we have $e=\bigvee e^{\Downarrow_{H}}$ and $e+g=g+\bigvee e^{\Downarrow_{H}}=\bigvee\left(g+e^{\Downarrow_{H}}\right)$. Now take $h\in g^{\Downarrow_{H}}$. As $\left.T\right|_{H}$ is linear, we have $T\left(h+e^{\Downarrow_{H}}\right)=Th+Te^{\Downarrow_{H}}$, so that
\begin{align*}
T\left(e+h\right)&=T\left(\bigvee\left(h+e^{\Downarrow_{H}}\right)\right)=\bigvee T\left(h+e^{\Downarrow_{H}}\right)\\
&=\bigvee \left(Th+Te^{\Downarrow_{H}}\right)=Th+\bigvee Te^{\Downarrow_{H}}=Th+ T\bigvee e^{\Downarrow_{H}}=Th+Te.
\end{align*}
Hence,
\begin{align*}
T\left(e+g\right)&=T\left(\bigvee\left(e+g^{\Downarrow_{H}}\right)\right)=\bigvee T\left(e+g^{\Downarrow_{H}}\right)=\bigvee\limits_{h\in g^{\Downarrow_{H}}}T\left(e+h\right)\\
&=\bigvee\limits_{h\in g^{\Downarrow_{H}}}\left(Te+Th\right)=Te+\bigvee Tg^{\Downarrow_{H}}=Te+ T\bigvee g^{\Downarrow_{H}}=Te+Tg.
\end{align*}

Positive homogeneity is proven similarly. An additive positively homogeneous map between vector spaces is always linear.\medskip

Now assume that $S:E\to F$ is a map such that $\left.S\right|_{H}=\left.T\right|_{H}$. If $S$ is supremum-continuous, then $Te=T\bigvee e^{\Downarrow_{H}}=\bigvee Te^{\Downarrow_{H}}=\bigvee Se^{\Downarrow_{H}}=S\bigvee e^{\Downarrow_{H}}=Se$, for every $e\in E$, hence $S=T$. If $S$ is positive, then by (ii) it is supremum-continuous, and so we are back to the previous case.
\end{proof}

\begin{theorem}\label{supcont1}
Let $E$ and $F$ be directed Archimedean ordered vector spaces. For an operator $T:E\to F$ the following conditions are equivalent:
\item[(i)] $T$ is supremum-continuous;
\item[(ii)] $T$ is positive and $T\left(A^{\Uparrow}\right)^{\Downarrow}\subset \left(TA\right)^{\Uparrow\Downarrow}$, for every nonempty bounded from above $A\subset E$;
\item[(iii)] $T\left(A^{\Uparrow\Downarrow}\right)\subset \left(TA\right)^{\Uparrow\Downarrow}$, for every nonempty bounded from above $A\subset E$;
\item[(iv)] There is a supremum-continuous operator $T^{\delta}:E^{\delta}\to F^{\delta}$ which extends $T$;
\item[(v)] There are Archimedean vector lattices $\widetilde{E},\widetilde{F}$ and a supremum-continuous operator $\widetilde{T}:\widetilde{E}\to \widetilde{F}$ such that $E$ embeds as a subspace of $\widetilde{E}$ in a supremum-continuous way, $F$ embeds as a subspace of $\widetilde{F}$, and $\widetilde{T}$ extends $T$.\medskip

Moreover, the extension in (iv) is unique even among positive operators. It is injective if and only if $T$ is an order embedding.
\end{theorem}

\begin{remark}
Note that $T$ is positive if and only if for every $A\subset E$ we have $T\left(A^{\Uparrow}\right)^{\Downarrow}\supset \left(TA\right)^{\Uparrow\Downarrow}$ if and only if for $A=\left\{0_{E}\right\}$ we have $T\left(A^{\Uparrow}\right)^{\Downarrow}\supset \left(TA\right)^{\Uparrow\Downarrow}$. Hence, the condition (ii) in Theorem \ref{supcont1} can be restated as $T\left(A^{\Uparrow}\right)^{\Downarrow}= \left(TA\right)^{\Uparrow\Downarrow}$, for every bounded from above $\varnothing\ne A\subset E$.

The conditions (ii) and (iii) in Theorem \ref{supcont1} strongly resemble some classes of operators considered in \cite[Section 2.3]{kv}, with the distinction that there $A$ runs only over finite nonempty subsets of $E$.
\qed\end{remark}

We recall that if $P$ is a poset, its \emph{MacNeille completion} $P^{\partial}$ is the set of all \emph{lower cuts} in $P$, i.e.  subsets $A$ of $P$ such that $A=A^{\Uparrow\Downarrow}$, ordered by inclusion.\footnote{We distinguish the concepts of Dedekind and MacNeille completions, and use different notations for them.} Note that $P^{\partial}$ is a complete lattice and we always view $P$ embedded into it via $p\mapsto p^{\Downarrow}$. If $\Bo\subset P^{\partial}$ (and so it is a collection of subsets of $P$), it is easy to see that $\bigvee_{P^{\partial}}\Bo=\left(\bigcup\Bo\right)^{\Uparrow\Downarrow}$. More information about $P^{\partial}$ see in e.g. \cite{erne} and the references therein. If $Q$ is also a poset and $\varphi:P\to Q$, define $\varphi^{\partial}:P^{\partial}\to Q^{\partial}$ by $\varphi^{\partial}\left(A\right):=\varphi\left(A\right)^{\Uparrow\Downarrow}$. Clearly, $\varphi^{\partial}$ is order preserving, and if $\varphi$ is also order preserving, then $\varphi^{\partial}$ extends $\varphi$. Furthermore, according to \cite{bishop}, $\varphi^{\partial}$ is supremum-continuous, also allowing for the empty set if and only if the pre-image of a lower cut under $\varphi$ is a lower cut. It is not hard to show (see e.g. \cite[Lemma 2.1(1)]{erne}) that this condition is equivalent to $\varphi\left(A^{\Uparrow\Downarrow}\right)\subset \left(\varphi\left(A\right)\right)^{\Uparrow\Downarrow}$, for every $A\subset P$. Let us prove a variation of this result.

\begin{lemma}\label{bish}
Let $P,Q$ be posets without greatest elements and let $\varphi:P\to Q$ be such that $\varphi\left(A^{\Uparrow\Downarrow}\right)\subset \left(\varphi\left(A\right)\right)^{\Uparrow\Downarrow}$, for every nonempty bounded above $A\subset P$. Then, $\varphi^{\partial}$ maps $P^{\partial}\backslash \left\{P,\varnothing\right\}$ into  $Q^{\partial}\backslash \left\{Q,\varnothing\right\}$ in a $\bigvee$-continuous way.
\end{lemma}
\begin{proof}
First, if $p\le r$, take $A:=\left\{r\right\}$. Then, $A^{\Uparrow\Downarrow}=r^{\Downarrow}$, hence $\varphi\left(p\right)\in \varphi\left(A^{\Uparrow\Downarrow}\right)\subset \left(\varphi\left(A\right)\right)^{\Uparrow\Downarrow}=\varphi\left(r\right)^{\Downarrow}$, and so $\varphi\left(p\right)\le \varphi\left(r\right)$. It follows that if $A\subset P$ is bounded above by some $p$, then $\varphi\left(A\right)$ is bounded above by $\varphi\left(p\right)$, and so $\varphi\left(A\right)^{\Uparrow\Downarrow}$ is bounded from above. Also, if $A\ne\varnothing$, then $\varphi^{\partial}\left(A\right)=\varphi\left(A\right)^{\Uparrow\Downarrow}\supset \varphi\left(A\right)\ne\varnothing$. Hence, $\varphi^{\partial}$ maps $P^{\partial}\backslash \left\{P,\varnothing\right\}$ into  $Q^{\partial}\backslash \left\{Q,\varnothing\right\}$.

Let $\varnothing\ne \Bo\subset P^{\partial}\backslash \left\{P,\varnothing\right\}$ and $B\in P^{\partial}\backslash \left\{P\right\}$ be such that $B=\bigvee \Bo$. Then $B=\left(\bigcup\Bo\right)^{\Uparrow\Downarrow}$, and since $B\in P^{\partial}\backslash \left\{P\right\}$, it is bounded from above, hence $\bigcup\Bo$ is bounded from above and nonempty. By our assumption
\begin{align*}
\varphi^{\partial}\left(B\right)&=\varphi\left(B\right)^{\Uparrow\Downarrow}=\varphi\left(\left(\bigcup\Bo\right)^{\Uparrow\Downarrow}\right)^{\Uparrow\Downarrow}= \varphi\left(\bigcup\Bo\right)^{\Uparrow\Downarrow}=\left(\bigcup\limits_{C\in\Bo}\varphi\left(C\right)\right)^{\Uparrow\Downarrow}\\
&\subset \left(\bigcup\limits_{C\in\Bo}\varphi\left(C\right)^{\Uparrow\Downarrow}\right)^{\Uparrow\Downarrow}
=\left(\bigcup\limits_{C\in\Bo}\varphi^{\partial}\left(C\right)\right)^{\Uparrow\Downarrow}=\bigvee\varphi^{\partial}\left(\Bo\right),
\end{align*}
implying $\varphi^{\partial}\left(B\right)=\bigvee\varphi^{\partial}\left(\Bo\right)$.
\end{proof}

If $E$ is a non-zero directed Archimedean OVS, then $E^{\delta}=E^{\partial}\backslash\left\{\varnothing, E\right\}$. The argument in (i)$\Rightarrow$(ii) is inspired by \cite[Proposition 2.3.16]{kv}.

\begin{proof}[Proof of Theorem \ref{supcont1}]
(i)$\Rightarrow$(ii): It is clear that $T$ is positive. For $B:=A^{\Uparrow}-A$ we have $\bigwedge B=0_{E}$ (the proof of \cite[Theorem 22.5]{lz} goes through), hence $\bigwedge T\left(B\right)=0_{F}$. Let $f\in T\left(A^{\Uparrow}\right)^{\Downarrow}$. For every $g\in \left(TA\right)^{\Uparrow}$, $e\in A^{\Uparrow}$ and $h\in A$ we have that $f\le Te$ and $g\ge Th$, therefore $f-g\le T\left(e-h\right)$. Since $e,h$ were arbitrary, it follows that $f-g\le T\left(B\right)$, hence $f-g\le 0_{F}$, and so $f\le g$. As $g$ was arbitrary, we conclude that $f\in \left(TA\right)^{\Uparrow\Downarrow}$.\medskip

(ii)$\Rightarrow$(iii): It is enough to show that $T\left(A^{\Uparrow\Downarrow}\right)\subset T\left(A^{\Uparrow}\right)^{\Downarrow}$, which is equivalent to $T\left(A^{\Uparrow\Downarrow}\right)\le T\left(A^{\Uparrow}\right)$. Indeed, if $e\in A^{\Uparrow\Downarrow}$ and $g\in A^{\Uparrow}$ then $e\le g$, hence $Te\le Tg$.\medskip

(iii)$\Rightarrow$(iv): First, the cases when $E=\left\{0_{E}\right\}$ or $F=\left\{0_{F}\right\}$ are trivial, and so we will assume that $\dim E,\dim F>0$. Then, $T^{\partial}$ maps $E^{\delta}$ into $F^{\delta}$ in a $\bigvee$-continuous way. This restriction $T^{\delta}$ extends $T$, and it is linear by part (iii) of Proposition \ref{line}.\medskip

(iv)$\Rightarrow$(v): The fact that $E$ embeds into $E^{\delta}$ in a supremum-continuous way follows from \cite[Section 2.1 and Proposition 2.3.27]{kv}. The rest is clear.\medskip

(v)$\Rightarrow$(i): Let $J:E\to \widetilde{E}$ be the inclusion, which is assumed supremum-continuous. It follows that $T=\widetilde{T}J$ is supremum-continuous as a map into $\widetilde{F}$. Let $\varnothing\ne P\subset E$ be such that $\bigwedge P=0_{E}$, so that $TP\subset F$ with $\bigwedge_{\widetilde{F}} TP=0_{\widetilde{F}}$, which implies $\bigwedge_{F} TP=0_{F}$. Thus, $T$ is supremum-continuous, as an operator into $F$.\medskip


Uniqueness of the extension follows from Proposition \ref{line}. If $T^{\delta}$ is injective, then it is an order embedding due to the fact that every injective semi-lattice homomorphism is such. It then follows that $T=\left.T^{\delta}\right|_{E}$ is an order embedding. Conversely, if $T$ is an order embedding, then by part (i) of Proposition \ref{line} the same is true for $T^{\delta}$, and in particular, the latter is injective.
\end{proof}

\begin{remark}
An alternative route of proving Theorem \ref{supcont1} is to use the result (see \cite[Theorem 3.1]{erne}) that for a map $\varphi:P\to Q$ between posets $P,Q$ the fact that $\varphi^{\partial}$ is supremum-continuous and infimum-continuous is equivalent to $\varphi\left(A^{\Uparrow}\right)^{\Downarrow}= \varphi\left(A\right)^{\Uparrow\Downarrow}$ and $\varphi\left(A^{\Downarrow}\right)^{\Uparrow}= \varphi\left(A\right)^{\Downarrow\Uparrow}$, for every $A\subset P$. However, this result is significantly more difficult than that from \cite{bishop}. Another alternative proof is to combine \cite[Corollary 11]{kvgs} with Veksler's theorem \cite[Theorem 1.65]{ab}.

Also, instead of using part (i) of Proposition \ref{line} we could rely on the fact (\cite[Theorem 20]{ac}; note that it is not used in the proof that the poset is a lattice) that if $\varphi$ is an order embedding, then so is $\varphi^{\partial}$.
\qed\end{remark}

\begin{remark}\label{exam}
The condition (iii) of Theorem \ref{supcont1} cannot be strengthened to the declared inclusion for every $A\subset E$. Indeed, it is enough to consider the projection $T$ onto the first coordinate in $\R^{2}$ and $A=\left\{0\right\}\times \left[0,+\8\right)$, since then $A^{\Uparrow\Downarrow}=\R^{2}$, while $TA=\left\{0\right\}$.

The condition (ii) also cannot be relaxed to only apply to order bounded $A$'s. Take the functional $\varphi$ from \cite[Example 26]{kvgs}, for $n=2$, which is not $\bigvee$-continuous. To prove that it satisfies the relaxed version of the condition (ii), take an order bounded $\varnothing\ne A\subset E$. Moreover, since $e:=\left(1,0,0\right)$ is a strong unit of $E$ we may assume that $A\subset\left[0_{E},e\right]$. Suppose that $r:=\bigvee\varphi\left(A\right)$. For $n\in\N$ let $f_{n}:=\frac{1}{2}\left(r+1+\frac{1}{n}+\frac{n}{4},r-1+\frac{1}{n}-\frac{n}{4},0\right)$; note that $\varphi\left(f_{n}\right)=r+\frac{1}{n}$.

Fix $\left(t,x,y\right)\in A\subset\left[0_{E},e\right]$. We have $1-t,t\ge \sqrt{x^{2}+y^{2}}$, hence $0\le t-x\le 1$, $y^{2}\le \frac{1}{4}$, and also $\varphi\left(a\right)=t+x\le r$. Let $\eta_{n}:=r+\frac{1}{n}-t-x\ge \frac{1}{n}$ and $\theta_{n}:=1+\frac{n}{4}-t+x\ge \frac{n}{4}$. Then, $f_{n}-a=\frac{1}{2}\left(\eta_{n}+\theta_{n},\eta_{n}-\theta_{n},-2y\right)$, and it is not hard to show that $f_{n}-a\in E_{+}$. Since $a$ was arbitrary, we conclude that $f_{n}\in A^{\Uparrow}$. It now follows that $\varphi\left(A^{\Uparrow}\right)^{\Downarrow}\subset \left\{\varphi\left(f_{n}\right)\right\}^{\Downarrow}=\left(-\8,r\right]=\left(\varphi\left(A\right)\right)^{\Uparrow\Downarrow}$.
\qed\end{remark}

\section{Free Dedekind and universally complete vector lattices}\label{fre}

Recall that if $K$ is a compact Hausdorff space, then $ C \left(K\right)$ is Dedekind complete if and only if $K$ is \emph{extremally disconnected}, i.e. $\overline{U}$ is open, for any open $U\subset K$ (see \cite[Theorem 1.50]{ab0}). If $K$ is extremally disconnected, let $ C ^{\8}\left(K\right)$ be the set of all continuous functions from $K$ to $\left[-\8,\8\right]$ such that $f^{-1}\left(\pm\8\right)$ is nowhere dense in $K$. If $f,g\in  C ^{\8}\left(K\right)$ there is a unique $h\in  C ^{\8}\left(K\right)$ such that $h\left(x\right)=f\left(x\right)+g\left(x\right)$, for every $x\notin f^{-1}\left(\pm\8\right)\cup g^{-1}\left(\pm\8\right)$; we declare this $h$ the sum of $f$ and $g$ in $ C ^{\8}\left(K\right)$. With this addition and its natural order and scalar multiplication $ C ^{\8}\left(K\right)$ is a universally complete vector lattice (see \cite[Theorem 7.27]{ab0}; recall that a Dedekind complete vector lattice $E$ is called \emph{universally complete} if every set of disjoint elements in $E$ is order bounded). If $E$ is an Archimedean vector lattice then there is an essentially unique universally complete vector lattice $E^{u}$ which contains $E$ as an order dense sublattice. Note that $E^{u}$ is isomorphic to $ C ^{\8}\left(K\right)$, where $K$ is the Stone space of the complete Boolean algebra of bands of $E$ (see \cite[Theorem 7.29]{ab0}). Spaces of this kind, also denoted $D\left(K\right)$, and for more general $K$'s, have been studied extensively, see e.g. \cite{bhjmm} and the references therein.\medskip

Let $\mathbf{AOVS}$, $\mathbf{DAOVS}$, $\mathbf{DVL}$, and $\mathbf{UVL}$ stand for the categories whose objects are Archimedean OVS's, directed Archimedean OVS's, Dedekind complete vector lattices, and universally complete vector lattices, respectively, and the morphisms are supremum-continuous operators.

If $\mathcal{C}$ is a category, we denote the class of its objects by $\left|\mathcal{C}\right|$, and if $C,D\in \left|\mathcal{C}\right|$ then $\mathcal{C}\left(C,D\right)$ stands for the set of morphisms from $C$ to $D$ (these are often denoted $Obj\left(\mathcal{C}\right)$ and $Hom\left(C,D\right)$).

For a category $\mathcal{C}$ whose objects are ``structured'' sets, and morphisms are maps between these sets, a \emph{free object} over a set $A$ is a pair $\left(C_{A},\iota\right)$, where $C_{A}\in\left|\mathcal{C}\right|$, $\iota:A\to C_{A}$ (here by $C_{A}$ we mean the underlying set) such that for every $D\in \left|\mathcal{C}\right|$ and any $\varphi:A\to D$ there is a unique $\widehat{\varphi}:C_{A}\to D$ such that $\widehat{\varphi}\circ\iota=\varphi$.

\begin{example}\label{cfree}
Let us discuss the notion of a free Dedekind complete vector lattice, i.e. a free object of $\mathbf{DVL}$ over a set $A$. First, it is easy to check that $\left\{0\right\}$ is the free Dedekind complete vector lattice over $\varnothing$. Let us now also show that $\left(\R^{2},\iota\right)$ is the free Dedekind complete vector lattice over a singleton $\left\{p\right\}$, where $\iota\left(p\right)=\left(1,-1\right)$. Let $F$ be a Dedekind complete vector lattice, and let $\varphi:\left\{p\right\}\to F$, so that $f:=\varphi\left(p\right)$. Let $T:\R^{2}\to F$ be defined by $T\left(r,s\right):=rf^{+}+sf^{-}$. Clearly, $T$ is linear and $T\circ\iota=\varphi$. We also have $T\left(\left|r\right|,\left|s\right|\right)=\left|r\right|f^{+}+\left|s\right|f^{-}=\left|rf^{+}+sf^{-}\right|=\left|T\left(r,s\right)\right|$, where the middle equality follows from the fact that $f^{+},f^{-}$ are positive and disjoint. Hence, $T$ is a homomorphism, and it is in fact order continuous (use the condition (iv) in \cite[Theorem 3.10]{erz} and the fact that any ideal in $\R^{2}$ is a band). Note that the same arguments show that $\left\{0\right\}$ and $\left(\R^{2},\iota\right)$ are the free universally complete vector lattices over $\varnothing$ and $\left\{p\right\}$, respectively.
\qed\end{example}

We will now show that Example \ref{cfree} describes the only free Dedekind complete vector lattices, using methods developed in Boolean algebra theory (see also \cite{hamana} for a similar application). We start with some auxiliary results.

\begin{proposition}\label{closed}
Let $E$ be a Dedekind complete vector lattice, let $F$ be a vector lattice, and let $T:E\to F$ be $\bigvee$-continuous such that $TE$ is majorizing in $F$. Then, $TE$ is a $\bigvee$-closed sublattice in $F$.
\end{proposition}
\begin{proof}
Assume that $f=\bigvee P$, where $P\subset TE$. Since $TE$ is majorizing, there is $e\in E$ such that $Te\ge f\ge P$. For every $p\in P$ there is $e_{p}\in E$ such that $Te_{p}=p$. As $T$ is a homomorphism, we have $T\left(e\wedge e_{p}\right)=Te\wedge p=p$, and so replacing $e_{p}$ with $e\wedge e_{p}$ if needed, we may assume that $e_{p}\le e$. Since $E$ is Dedekind complete, there is $g:=\bigvee\limits_{p\in P}e_{p}$, and then due to $\bigvee$-continuity of $T$ we get $Tg=\bigvee\limits_{p\in P}Te_{p}=\bigvee\limits_{p\in P}p=f$. Thus, $f\in TE$.
\end{proof}

\begin{remark}
If $TE$ is not majorizing, then $TE$ may fail to be $\bigvee$-closed in $F$. For example, consider $E:=c_{0}$, $F:=\ell_{\8}$ and $T:E\to F$ -- the inclusion. Then, $\1=\bigvee\limits_{n\in\N}e_{n}$, where $\left\{e_{n}\right\}_{n\in\N}$ is the standard basis of $c_{0}$.
\qed\end{remark}

For $e\in E$ we denote by $I_{e}$ the smallest ideal of $E$ containing $e$, i.e. $I_{e}=\bigcup\limits_{r>0}\left[-r\left|e\right|,r\left|e\right|\right]$.

\begin{lemma}\label{closedu}
Let $E$ be a Dedekind complete vector lattice, let $F$ be a vector lattice, and let $T:E\to F$ be $\bigvee$-continuous. For every $f\in TE$ the sublattice $TE\cap I_{f}$ is $\bigvee$-closed in $I_{f}$.
\end{lemma}
\begin{proof}
Let $H:=T^{-1}I_{f}$, which is an ideal in $E$, and thus a Dedekind complete vector lattice. It is easy to see that $\left.T\right|_{H}$ is $\bigvee$-continuous into $I_{f}$, and since $f\in TH$, Proposition \ref{closed} guarantees that $TH=TE\cap I_{f}$ is a $\bigvee$-closed sublattice of $I_{f}$.
\end{proof}

The following result is analogous to \cite[Theorems 13.1]{koppel} (see also \cite{solovay}).

\begin{proposition}\label{soloway}
For any cardinal $\kappa$ there is an extremally disconnected compact Hausdorff space $L$ and $u\in C \left(L\right)_{+}$ such that $\card C \left(L\right)\ge\kappa$ and there is no proper $\bigvee$-closed sublattice of $ C \left(L\right)$ containing $\1_{L},u$.
\end{proposition}
\begin{proof}
According to \cite[Theorems 13.1 and 1.37]{koppel} there is a complete Boolean algebra $B$ with $\card B\ge\kappa$ and a countable $C\subset B$ such that no proper $\bigvee$-closed subalgebra of $B$ contains $C$. Let $L$ be the Stone space of $B$, which is extremally disconnected (see \cite[Theorem 7.21]{koppel}).

Let $K\subset\left[0,1\right]$ be the standard Cantor set, and let $g:K\to\R$ be the inclusion map. It follows from the Stone-Weierstrass theorem (see e.g. \cite[Theorem 4.4]{erz}) that the vector lattice $E$ generated by $\1_{K},g$ is norm dense in $ C \left(K\right)$. Let $A$ be the Boolean algebra of clopen sets of $K$, which is the free Boolean algebra over $\N$ (see \cite[Corollary 9.7(a)]{koppel}). Hence, there is a Boolean homomorphism $\Phi:A\to B$ whose range contains $C$. This homomorphism is induced by a continuous map $\varphi:L\to K$ via $\Phi\left(U\right)=\varphi^{-1}\left(U\right)$ (see \cite[Theorem 8.2]{koppel}). In turn, $\varphi$ generates a unital vector lattice homomorphism $J: C \left(K\right)\to C \left(L\right)$ via $f\mapsto f\circ\varphi$. Note that if we identify $A$ with the Boolean algebra $\left\{\1_{U},~U\subset K\mbox{ -- clopen}\right\}\subset C \left(K\right)$, and analogously for $B$, then $\Phi=\left.J\right|_{A}$. Also, note that $\card  C \left(L\right)\ge \card B\ge\kappa$.

Let $u:=Jg$, and note that $\1_{L}=J\1_{K}$. Assume that $F\subset  C \left(L\right)$ is a $\bigvee$-closed sublattice which contains $\1_{L},u$. Then, $F$ contains $JE$. Fix $c\in C$ and let $a\in A$ be such that $c=\Phi\left(a\right)$. Note that we view $A\subset C \left(K\right)$ and $B\subset C \left(L\right)$. Since $E$ is norm dense in $ C \left(K\right)$, for every $n\in\N$ there is $e_{n}\in E$ such that $\|a-e_{n}\|\le\frac{1}{n}$, so that $\left|a-e_{n}\right|\le\frac{1}{n} \1_{K}$. It follows that $\left|Ja-Je_{n}\right|\le\frac{1}{n}\1_{L}$. Since $F$ is $\bigvee$-closed, it is Dedekind complete. Hence, it is  relatively uniformly complete (see \cite[Lemma 1.56]{ab0}). As it also contains $\1_{L}$, the uniform convergence in both $ C \left(L\right)$ and $F$ are given by the supremum norm, and so $F$ is norm closed in $ C \left(L\right)$. Thus, $c=Ja\in F$.

Therefore, $C\subset F$; note that the set of components of $\1_{L}$ in $F$ is a $\bigvee$-closed subalgebra of $B$ (see \cite[Theorem 1.53(5)]{ab0}), which contains $C$, hence equals $B$ by construction. Thus, $F$ is norm closed in $ C \left(L\right)$ and contains the characteristic functions of all clopen sets in $L$. Since the latter separate points of $L$, by Stone-Weierstrass theorem we conclude that $F= C \left(L\right)$.
\end{proof}

The following result is analogous to \cite[Corollary 13.2]{koppel} (see also \cite{gaifman} and \cite{hales}) and improves \cite[Theorem 1]{jakubikova} (see also \cite{jak1} and \cite{jak2}).

\begin{proposition}\label{nofree}
There are no free Dedekind complete or universally complete vector lattices over a set with more than one element.
\end{proposition}
\begin{proof}
Assume that $A$ contains distinct elements $p,q$. Let $\left(F,\iota\right)$ be a free Dedekind complete vector lattice over $A$. Let $\kappa$ be larger than $\card F$, and let $L$ and $u\in  C \left(L\right)_{+}$ be given via Proposition \ref{soloway}. Let $\varphi:A\to C \left(L\right)$ be an arbitrary map such that $\varphi\left(p\right)=\1_{L}$ and $\varphi\left(q\right)=u$. By assumption, there is a $\bigvee$-continuous operator $\widehat{\varphi}:F\to  C \left(L\right)$ such that $\widehat{\varphi}\circ\iota=\varphi$. In particular, $\1_{L},u\in G:=\widehat{\varphi}F$, and so $G$ is majorizing in $ C \left(L\right)$. By Proposition \ref{closed}, it follows that $G$ is $\bigvee$-closed in $ C \left(L\right)$. Hence, by construction it must be equal to $ C \left(L\right)$. It follows that $\widehat{\varphi}$ is a surjection, but it contradicts the assumption that $\card  C \left(L\right)\ge\kappa> \card F$.\medskip

Now assume that $\left(F,\iota\right)$ is a free universally complete vector lattice over $A$. Arguing as before, we produce a $\bigvee$-continuous operator $\widehat{\varphi}:F\to  C ^{\8}\left(L\right)$ such that $\1_{L},u\in \widehat{\varphi}F$. Then, by Lemma \ref{closedu} $\widehat{\varphi}F\cap  C \left(L\right)$ is a $\bigvee$-closed sublattice of $ C \left(L\right)$. Hence, $\widehat{\varphi}F\cap  C \left(L\right)=  C \left(L\right)$, and so we reach a contradiction in the same way as before.
\end{proof}

Let $\mathbf{AVL}$ be the full subcategory of $\mathbf{DAOVS}$ whose objects are vector lattices. If $E$ is a free directed Archimedean ordered vector space over a set $A$, then it follows from Theorem \ref{supcont1} that $E^{\delta}$ is a free Dedekind complete vector lattice over $A$. Combining this with Proposition \ref{nofree} we get the following result.

\begin{corollary}
There are no free objects in either $\mathbf{DAOVS}$ or $\mathbf{AVL}$ over a set with more than one element.
\end{corollary}

We will now discuss a somewhat related topic. Let $E$ be an Archimedean vector lattice. A \emph{free completion} of $E$ is a pair $\left(\widehat{E},\iota\right)$, where $\widehat{E}$ is a Dedekind complete vector lattice and $\iota:E\to \widehat{E}$ is a vector lattice homomorphism, such that whenever $T:E\to F$ is a vector lattice homomorphism into a Dedekind complete vector lattice $F$, there is a unique supremum-continuous map $\widehat{T}:\widehat{E}\to F$ such that $T=\widehat{T}\circ\iota$. It is not hard to show that the free completion is unique once it exists. A similar definition can be given in the category of Boolean algebras, where the following result was proven in \cite{day}.

\begin{theorem}\label{day}
A Boolean algebra $A$ has a free Boolean algebra completion $\widehat{A}$ if and only if $A$ is superatomic (i.e. every subalgebra of $A$ is atomic). In this case $A\simeq Clop\left(K\right)$, where $K$ is scattered, and $\widehat{A}\simeq\Po\left(K\right)$, the powerset of $K$.
\end{theorem}

Recall that a topological space is called \emph{scattered} if every subset has an isolated point. Let us reinterpret the second part of Theorem \ref{day} in terms of Stone duality (see \cite[Chapter 3]{koppel}). First, note that $\Po\left(K\right)\simeq Clop\left(\beta K_{d}\right)$, where $K_{d}$ is the set $K$ endowed with the discrete topology and $\beta K_{d}$ is its Stone-Cech compactification. The inclusion of $Clop \left(K\right)$ into $\Po\left(K\right)$ corresponds to the extension $p:\beta K_{d}\to K$ of the identity map on $K$. Also, $Clop\left(K\right)$ is superatomic if and only if $K$ is scattered (see \cite[Remark 17.2]{koppel}) and complete if and only if $K$ is extremally disconnected (see \cite[Proposition 7.21]{koppel}). Finally, the continuous maps corresponding to complete (supremum continuity) Boolean homomorphisms have the property that the pre-images of nowhere dense closed sets are nowhere dense\footnote{The converse is also true, but we will not need this, and could not find a reference for it.} (this follows from \cite[Exercise 313X(q)]{fremlin}); for a continuous map into $\beta K_{d}$ this simply means that the pre-image of $\beta K_{d}\backslash K$ is nowhere dense. Thus, we get the following topological result, which we also supply with a direct topological proof.



\begin{corollary}\label{day2}
Let $K$ and $M$ be compact Hausdorff spaces, with $K$ scattered and $M$ extremally disconnected. If $\varphi:M\to K$ is continuous, there is a unique continuous map $\widehat{\varphi}:M\to\beta K_{d}$ such that $\widehat{\varphi}^{-1}\left(\beta K_{d}\backslash K\right)$ is nowhere dense in $M$, and $\varphi=p\circ \widehat{\varphi}$.
\end{corollary}
\begin{proof}
Since $K$ is scattered, there is an ordinal $\alpha_{0}$ such that the $\alpha_{0}$-th derivative $K_{\alpha_{0}}$ (see \cite[Construction 17.7]{koppel}) of $K$ is empty (see \cite[Proposition 17.8]{koppel}). Fix $\alpha\le\alpha_{0}$.

Denote the set of isolated points of $K_{\alpha}$ by $K^{\alpha}$ (so that $K_{\alpha+1}=K_{\alpha}\backslash K^{\alpha}$); also, let $M_{\alpha}:=\varphi^{-1}\left(K_{\alpha}\right)$ and $M^{\alpha}:=\varphi^{-1}\left(K^{\alpha}\right)$. Note that $M_{\alpha}$ is closed in $M$, $M^{\alpha}$ is open in $M_{\alpha}$ and $K_{\alpha}=K\backslash \bigcup\limits_{\beta<\alpha}K^{\beta}$, which implies that $M_{\alpha}=M\backslash \bigcup\limits_{\beta<\alpha}M^{\beta}$.

Define also $U_{\alpha}:=M^{\alpha}\backslash \overline{\bigcup\limits_{\beta<\alpha}M^{\beta}}=M^{\alpha}\cap \Int M_{\alpha}$. Note that since $M^{\alpha}$ is open in $M_{\alpha}$, it follows that $U_{\alpha}$ is open in $M$. Also, since $M_{\alpha}$ is closed and $M$ is extremally disconnected, we get that $\Int M_{\alpha}$ is clopen, and so $\overline{U_{\alpha}}=\overline{M^{\alpha}}\cap \Int M_{\alpha}=\overline{M^{\alpha}}\backslash \overline{\bigcup\limits_{\beta<\alpha}M^{\beta}}$.

Let us show that $\bigcup\limits_{\beta\le\alpha}U_{\beta}$ is dense in $\bigcup\limits_{\beta\le\alpha}M^{\beta}$. The case $\alpha=0$ follows from $U_{0}=M^{0}$. Assume that the claim is proven for every $\alpha'<\alpha$. Then, $\overline{\bigcup\limits_{\beta<\alpha}U_{\beta}}$ is a closed set which is dense in $M^{\beta}$, for every $\beta<\alpha$, and so it contains $\overline{\bigcup\limits_{\beta<\alpha}M^{\beta}}$. Therefore, $\overline{\bigcup\limits_{\beta\le\alpha}U_{\beta}}=\overline{\bigcup\limits_{\beta<\alpha}U_{\beta}}\cup \overline{U_{\alpha}}\supset \overline{\bigcup\limits_{\beta<\alpha}M^{\beta}}\cup \overline{M^{\alpha}}\backslash \overline{\bigcup\limits_{\beta<\alpha}M^{\beta}}=\overline{\bigcup\limits_{\beta\le\alpha}M^{\beta}}$.

In particular, it follows that $U:=\bigcup\limits_{\beta\le\alpha_{0}}U_{\beta}$ is dense in $\bigcup\limits_{\beta\le\alpha_{0}}M^{\beta}=M\backslash M_{\alpha_{0}}=M$. Let us show that $\left.\varphi\right|_{U}$ is continuous into $K_{d}$. If $x\in K$ there is $\alpha\le\alpha_{0}$ such that $x\in K^{\alpha}$. Then, $\varphi^{-1}\left(x\right)\cap U=\varphi^{-1}\left(x\right)\cap U\cap M^{\alpha}=\varphi^{-1}\left(x\right)\cap U_{\alpha}$. Since $\varphi$ continuously maps $U_{\alpha}$ into $K^{\alpha}$, and $x$ is an isolated point in the latter, we conclude that $\varphi^{-1}\left(x\right)$ is open in $U_{\alpha}$, and therefore open in $U$.

It follows that $\left.\varphi\right|_{U}$ is continuous into $\beta K_{d}$. As the latter is compact, and $U$ is a dense open set in an extremally disconnected space $M$, $\left.\varphi\right|_{U}$ can be extended to a continuous $\widehat{\varphi}:M\to \beta K_{d}$ (this follows from \cite[Lemma 7.25]{ab0} and the fact that $\beta K_{d}$ can be embedded into a Tychonoff cube of a high enough dimension). Clearly, $\widehat{\varphi}^{-1}\left(\beta K_{d}\backslash K\right)\subset M\backslash U$ is nowhere dense in $M$. Moreover, since $p\circ \widehat{\varphi}$ agrees with $\varphi$ on $U$, by continuity we get $p\circ \widehat{\varphi}=\varphi$. 

Assume that $\psi:M\to\beta K_{d}$ is continuous and such that $\psi^{-1}\left(\beta K_{d}\backslash K\right)$ is nowhere dense in $M$, and $\varphi=p\circ \psi$. Then, $V:=M\backslash \left(\widehat{\varphi}^{-1}\left(\beta K_{d}\backslash K\right)\cup \psi^{-1}\left(\beta K_{d}\backslash K\right)\right)$ is dense in $M$, and since $p$ is the identity map on $K$, we conclude that $\widehat{\varphi}$ and $\psi$ agree on $V$, hence $\psi=\widehat{\varphi}$.
\end{proof}

The following result is a partial analogue of Theorem \ref{day} for vector lattices (see also \cite{macula} where the free completion was constructed in the class of linear maps which preserve suprema of sets of bounded cardinality).

\begin{proposition}\label{day3}
Let $E$ be an Archimedean vector lattice with a strong unit. Then, $E$ has a free completion $\widehat{E}$ if and only if $E$ embeds as a norm-dense sublattice of $ C \left(K\right)$, where $K$ is scattered. In this case $\widehat{E}\simeq\ell_{\8}\left(K\right)$.
\end{proposition}
\begin{proof}
Since $E$ has a strong unit we may assume that $E$ is a norm-dense sublattice of $ C \left(K\right)$, for some compact Hausdorff space $K$, and $\1_{K}\in E$ (see \cite[Theorem 45.3]{lz}). Our goal is to show that there is a free completion $\left(\widehat{E},\iota\right)$ of $E$ if and only if $K$ is scattered.\medskip

Necessity: Assume that $K$ is not scattered, and let $\kappa$ be greater than the cardinality of $\widehat{E}$. There is a continuous surjection $\varphi:K\to\left[0,1\right]$ (see \cite[Theorem 8.5.4]{semadeni}). The composition operator generated by $\varphi$ is an isomorphic embedding of $ C \left[0,1\right]$ into $ C \left(K\right)$ as a majorizing sublattice.

Let $L$ and $u\in C \left(L\right)$ be given by Proposition 3.4; note that $u\left(L\right)\subset \left[0,1\right]$. The composition operator generated by $u$ is a vector lattice homomorphism $T: C \left[0,1\right]\to  C \left(L\right)$. Note that $T$ maps the identity function on $\left[0,1\right]$ into $u$. As $ C \left[0,1\right]$ is a majorizing sublattice of $ C \left(K\right)$, and $ C \left(L\right)$ is Dedekind complete, $T$ can be extended to a vector lattice homomorphism from $ C \left(K\right)$ into $ C \left(L\right)$ (see \cite[Theorem 2.29]{ab}).

By assumption there is a supremum-continuous operator $\widehat{T}:\widehat{E}\to  C \left(L\right)$ such that $\left.T\right|_{E}=\widehat{T}\circ\iota$. Since $\1_{K}\in E$, it follows that $\1_{L}=T\1_{K}=\widehat{T}\left(\iota\left(\1_{K}\right)\right)\in \widehat{T}\widehat{E}$. Arguing as in the proof of Proposition \ref{soloway} we see that $\widehat{T}\widehat{E}$ is norm-closed in $ C \left(L\right)$. As $T$ is continuous, $E$ is norm-dense in $ C \left(K\right)$, and $TE\subset \widehat{T}\widehat{E}$, we conclude that $u\in T C \left(K\right)\subset \widehat{T}\widehat{E}$. Arguing as in the proof of Proposition \ref{nofree} we can achieve the desired contradiction.\medskip

Sufficiency: Assume that $F$ is a Dedekind complete vector lattice and $T: E\to F$ is a homomorphism. It is easy to see that the range of $T$ is contained in the principal ideal of $T\1_{K}$, and so we may assume that $F$ is equal to that principal ideal (which is still Dedekind complete). In fact, combining \cite[Theorem 45.4]{lz} with \cite[Theorem 1.50]{ab0} allows us to assume without loss of generality that $F\simeq C \left(M\right)$, for some extremally disconnected compact Hausdorff space $M$, and $T\1_{K}=\1_{M}$. Next, we can extend $T$ by continuity in a unique way, and assume that $E=\Co\left(K\right)$. Hence, there is a continuous $\varphi:M\to K$ such that $Tf=f\circ\varphi$, for every $f\in  C \left(K\right)$ (see \cite[Exercise 2.6]{ab}).

Let $\widehat{\varphi}:M\to\beta K_{d}$ be given by Corollary \ref{day2}, and let $\widehat{T}: C \left(\beta K_{d}\right)\to  C \left(M\right)$ be given via $\widehat{T}g:=g\circ\widehat{\varphi}$, for $g\in  C \left(\beta K_{d}\right)$. As the pre-images of nowhere dense closed sets in $\beta K_{d}$ are nowhere dense in $M$, it follows from \cite[Theorem 7.1(iii) and Proposition 6.3]{erz} that $\widehat{T}$ is supremum-continuous. Note that $ C \left(\beta K_{d}\right)\simeq \ell_{\8}\left(K\right)$, and the inclusion of $ C \left(K\right)$ into $\ell_{\8}\left(K\right)$ via this identification corresponds to $\iota: C \left(K\right)\to  C \left(\beta K_{d}\right)$ given by the composition with $p:\beta K_{d}\to K$, as in Corollary \ref{day2}. Thus, we have $T=\widehat{T}\circ\iota$, as required. Uniqueness of $\widehat{T}$ follows from the uniqueness of $\widehat{\varphi}$, as well as uniqueness of the extension by continuity.
\end{proof}

One may wonder whether the fact that $ C \left(K\right)^{**}=\ell_{\8}\left(K\right)$, for scattered $K$ (see \cite[Corollary 19.7.7]{semadeni}) can be used in the proof.

\begin{question}
Is there a result similar to Proposition \ref{day3} for Archimedean vector lattices which do not necessarily have a strong unit?
\end{question}

\section{Dedekind and universal completions as reflectors}

Recall that a category $\mathcal{D}$ is a \emph{subcategory} of a category $\mathcal{C}$ if $\left|\mathcal{D}\right|\subset \left|\mathcal{C}\right|$, and for every $D_{1},D_{2}\in \left|\mathcal{D}\right|$ we have $\mathcal{D}\left(D_{1},D_{2}\right)\subset\mathcal{C}\left(D_{1},D_{2}\right)$. If the latter inclusion is an equality for all $D_{1},D_{2}$, we say that $\mathcal{D}$ is a \emph{full subcategory} of $\mathcal{C}$.

Let $\mathcal{D}$ be a full subcategory of a category $\mathcal{C}$ and let $C\in \left|\mathcal{C}\right|$. A \emph{reflector} of $C$ in $\mathcal{D}$ is a pair $\left(C^{\mathcal{D}},\iota_{C}\right)$, where $C^{\mathcal{D}}\in \left|\mathcal{D}\right|$ and a morphism $\iota_{C}\in\mathcal{C}\left(C,C^{\mathcal{D}}\right)$ are such that for every $D\in\left|\mathcal{D}\right|$ and $\varphi\in \mathcal{C}\left(C,D\right)$ there is a unique $\varphi^{\mathcal{D}}\in\mathcal{D}\left(C^{\mathcal{D}},D\right)$ such that
$\varphi=\varphi^{\mathcal{D}}\circ\iota_{C}$. The standard ``diagram chasing'' argument shows that the reflector is essentially unique. If $\iota_{C}$ is clear from the context, we will simply refer to $C^{\mathcal{D}}$ as the reflector.

If $C_{1},C_{2}$ both have reflectors $C^{\mathcal{D}}_{1},C^{\mathcal{D}}_{2}$, respectively, and $\varphi\in\mathcal{C}\left(C_{1},C_{2}\right)$ then there is a unique $\varphi^{\mathcal{D}}\in\mathcal{D}\left(C^{\mathcal{D}}_{1},C^{\mathcal{D}}_{2}\right)$ such that $\varphi^{\mathcal{D}}\circ\iota_{C_{1}}=\iota_{C_{2}}\circ\varphi$. If $C^{\mathcal{D}}$ is a reflector of $C$ in $\mathcal{D}$, and $C^{\mathcal{E}}$ is a reflector of $C^{\mathcal{D}}$ in a full subcategory $\mathcal{E}$ of $\mathcal{D}$, then $C^{\mathcal{E}}$ is a reflector of $C$ in $\mathcal{E}$.

A full subcategory $\mathcal{D}$ of a category $\mathcal{C}$ is called \emph{reflective} if every object $C\in \left|\mathcal{C}\right|$ has a reflector in $\mathcal{D}$. It follows from the comment above that being a reflective subcategory is a transitive relation.\medskip

Note that among $\mathbf{AOVS}$, $\mathbf{DAOVS}$, $\mathbf{DVL}$, and $\mathbf{UVL}$ each next entry is a full subcategory of the previous one. In fact, more is true.

\begin{proposition}\label{universal}
\item[(i)] $\mathbf{DVL}$ is a reflective subcategory of $\mathbf{DAOVS}$ with the Dedekind completion being the reflector.
\item[(ii)] $\mathbf{UVL}$ is a reflective subcategory of $\mathbf{DVL}$ with the universal completion being the reflector.
\end{proposition}
\begin{proof}
(i) follows from Theorem \ref{supcont1} applied to the case when $F$ is Dedekind complete. (ii) is essentially a partial reformulation of \cite[Theorem 7.17]{ab0} or \cite[Theorem 4]{al}.
\end{proof}

It follows that $\mathbf{UVL}$ is a reflective subcategory of $\mathbf{DAOVS}$, with $E\mapsto E^{\delta u}$ being the reflector. We propose to call $E^{\delta u}$ the \emph{universal completion} of a directed Archimedean ordered vector space $E$. Note that this is consistent with the notion considered in \cite{stennder}.\medskip

We will now investigate if any non-directed Archimedean vector spaces have reflectors in $\mathbf{DVL}$ or $\mathbf{UVL}$.

It is clear that the only ordered vector space of dimension $0$ is $\left\{0\right\}$. It is also easy to see that the only two ordered vector spaces of dimension $1$ are $\R$ with the standard order and $\R$ with the positive cone containing only $0$ (we denote such a vector space by $\R_{0}$). Note that $\R_{0}$ is Archimedean.

\begin{theorem}\label{main}
For an Archimedean ordered vector space $E$ the following conditions are equivalent:
\item[(i)] Either $E$ is directed, or $E\simeq\R_{0}$;
\item[(ii)] $E$ has a reflector in $\mathbf{DVL}$;
\item[(iii)] $E$ has a reflector in $\mathbf{UVL}$.
\end{theorem}
\begin{proof}
(i)$\Rightarrow$(ii): The case when $E$ is directed is covered by part (i) of Proposition \ref{universal}. Assume that $E\simeq\R_{0}$. Recall that Example \ref{cfree} has established that $\left(\R^{2},\iota\right)$ is a free Dedekind complete vector lattice over $\left\{1\right\}$, where $\iota\left(1\right)=\left(1,-1\right)$. Let $J:\R_{0}\to \R^{2}$ be the linear extension of $\iota$, i.e. $Jt=\left(t,-t\right)$, for every $t\in\R$. Let us show that $\left(\R^{2},J\right)$ is the desired reflector. Assume that $F$ is a Dedekind complete vector lattice, and $T:\R_{0}\to F$ is linear. There is a $\bigvee$-continuous linear map $\widehat{T}:\R^{2}\to F$ such that $\widehat{T}\left(1,-1\right)=T1$. Then, $\widehat{T}Jt=\widehat{T}\left(t,-t\right)=t\widehat{T}\left(1,-1\right)=tT1=Tt$, for every $t\in\R$.\medskip

(ii)$\Rightarrow$(iii) follows from part (ii) of Proposition \ref{universal}, and the fact that a reflector of a reflector is a reflector.\medskip

(iii)$\Rightarrow$(i): Assume that $\left(F,\iota\right)$ is the reflector.

We first consider the case when $E_{+}= \left\{0_{E}\right\}$. Let $A$ be a Hamel basis of $E$, and let $j:A\to E$ be the inclusion map. Let us show that $\left(F,\iota\circ j\right)$ is a free universally complete vector lattice over $A$. Indeed, if $G$ is a universally complete vector lattice, and $\varphi:A\to G$, there is a unique linear operator $T:E\to G$ which extends $\varphi$. By Example \ref{trivial} $T$ is $\bigvee$-continuous, and by assumption there is a unique $\bigvee$-continuous $T^{u}:F\to G$ such that $T=T^{u}\iota$, which implies $\varphi=T\circ j=T^{u}\iota\circ j$. If $S:F\to G$ is $\bigvee$-continuous and such that $\varphi=S\iota\circ j=T\circ j$, then since $A$ is a Hamel basis, it follows that $S\iota=T$, and since $T^{u}$ is uniquely determined by $T$, we conclude that $S=T^{u}$. Now, according to Proposition \ref{nofree} the cardinality of $A$ is at most $1$, and so $\dim E\le 1$. This means that either $E\simeq\R_{0}$ or $\dim E=0$ (in which case it is directed).\medskip

We now assume that $E_{+}\ne \left\{0_{E}\right\}$. The goal is to show that $E$ is directed. Let $H:=E_{+}-E_{+}$, and let $G$ be its algebraic complement. Assume towards contradiction that $\dim G>0$. Let $I_{H}:H\to E$ be the inclusion map, which is $\bigvee$-continuous, by Example \ref{bigv}. There is a compact Hausdorff extremally disconnected space $K$ such that $ C _{\8}\left(K\right)$ is the universal completion of $H$. Let $J_{H}:H\to  C _{\8}\left(K\right)$ be the corresponding embedding, which is $\bigvee$-continuous. Since $\dim H>0$, $K$ is nonempty. By the universal property of $H^{u}= C _{\8}\left(K\right)$ there is a $\bigvee$-continuous operator $\left(\iota\circ I_{H}\right)^{u}: C ^{\8}\left(K\right)\to F$ such that $\left(\iota\circ I_{H}\right)^{u}\circ J_{H} = \iota\circ I_{H}$. Additionally, let $J_{ C \left(K\right)}: C \left(K\right)\to  C _{\8}\left(K\right)$ be the inclusion, which is also $\bigvee$-continuous.\medskip

Let $\kappa$ be larger than $\card F$, and let $L$ and $u\in  C \left(L\right)_{+}$ be given via Proposition \ref{soloway}. Let $R_{K}: C \left(K\right)\to C \left(K\times L\right)$ be given via $\left[R_{K}f\right]\left(x,y\right):=f\left(x\right)$, $f\in  C \left(K\right)$, $x\in K$, $y\in L$. Clearly, $R_{K}$ is an injection; moreover, it is $\bigvee$-continuous (this can be easily verified directly, or deduced from \cite[Theorem 4.4]{imhoff} or \cite[Theorem 7.1(iii)]{erz}, since $R_{K}$ is the composition operator with respect to the coordinate projection from $K\times L$ onto $K$; note that a pre-image of a nowhere dense set in $K$ is nowhere dense). An analogously defined operator $R_{L}: C \left(L\right)\to C \left(K\times L\right)$ is also a $\bigvee$-continuous injection. Note that $R_{K}\1_{K}=\1_{K\times L}=R_{L}\1_{L}$.

Let $J: C \left(K\times L\right)\to  C \left(K\times L\right)^{u}$ be the embedding, which is $\bigvee$-continuous. By the universal property of $ C \left(K\right)^{u}= C _{\8}\left(K\right)$ there is a unique $\bigvee$-continuous operator $R_{K}^{u}: C ^{\8}\left(K\right)\to  C \left(K\times L\right)^{u}$ such that $R_{K}^{u} J_{ C \left(K\right)} = JR_{K}$. In particular, the range of $R_{K}^{u}$ contains that of $JR_{K}$.

Let $S:G\to  C \left(L\right)$ be an arbitrary linear map such that $u\in SG$. Define $T:E\to  C \left(K\times L\right)^{u}$ by $$T\left(h+g\right):=R_{K}^{u}J_{H}h+JR_{L}Sg,\mbox{ where }h\in H,~g\in G.$$ Since $R_{K}^{u}J_{H}:H\to  C \left(K\times L\right)^{u}$ is $\bigvee$-continuous, according to Example \ref{bigv} so is $T$. Therefore, by the assumption there is a $\bigvee$-continuous $\widehat{T}:F\to  C \left(K\times L\right)^{u}$ such that $T=\widehat{T}\circ\iota$.
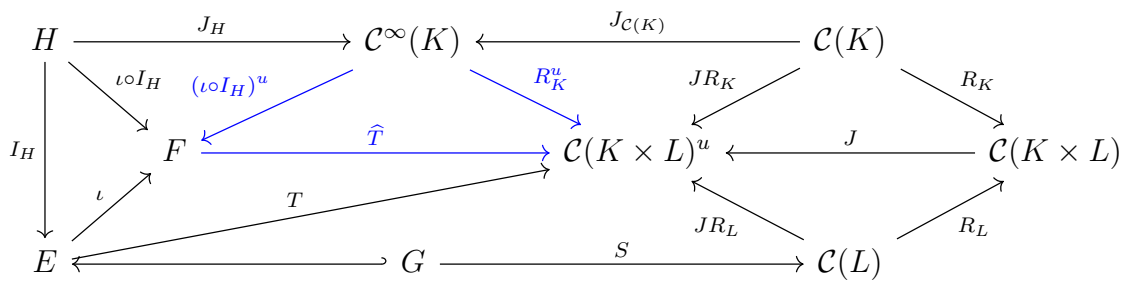
$$\begin{tikzcd}
	H &&& {\mathcal{C}^{\infty}(K)} && {\mathcal{C}(K)} \\
	& F &&& {\mathcal{C}(K\times L)^{u}} && {\mathcal{C}(K\times L)} \\
	E &&& G && {\mathcal{C}(L)}
	\arrow["{J_{H}}", from=1-1, to=1-4]
	\arrow["{\iota\circ I_{H}}", from=1-1, to=2-2]
	\arrow["{I_{H}}"', from=1-1, to=3-1]
	\arrow["{\left(\iota\circ I_{H}\right)^{u}}"', color={blue}, from=1-4, to=2-2]
	\arrow["{R_{K}^{u}}", color={blue}, from=1-4, to=2-5]
	\arrow["{J_{\mathcal{C}(K)}}"', from=1-6, to=1-4]
	\arrow["JR_{K}"', from=1-6, to=2-5]
	\arrow["{R_{K}}", from=1-6, to=2-7]
	\arrow["{\widehat{T}}", color={blue}, from=2-2, to=2-5]
	\arrow["J"', from=2-7, to=2-5]
	\arrow["\iota", from=3-1, to=2-2]
	\arrow["T", from=3-1, to=2-5]
	\arrow[hook', from=3-4, to=3-1]
	\arrow["S", from=3-4, to=3-6]
	\arrow["JR_{L}", from=3-6, to=2-5]
	\arrow["{R_{L}}"', from=3-6, to=2-7]
\end{tikzcd}$$

Clearly, $JR_{L}u\in JR_{L}SG \subset TE\subset \widehat{T}F$. On the other hand, $\widehat{T}\left(\iota\circ I_{H}\right)^{u}J_{H}=\widehat{T}\iota \circ I_{H}=TI_{H}=R_{K}^{u}J_{H}$, which implies $\widehat{T}\left(\iota \circ I_{H}\right)^{u}=R_{K}^{u}$ (the blue part of the diagram above commutes). In particular, the range of $\widehat{T}$ contains that of $R_{K}^{u}$, which in turn contains the range of $JR_{K}$. Therefore, $f:=JR_{L}\1_{L}=J\1_{K\times L}=JR_{K}\1_{K}\in \widehat{T}F$.

According to Lemma \ref{closedu} $\widehat{T}F\cap I_{f}$ is a $\bigvee$-closed sublattice of $I_{f}$. Note that $JR_{L} C \left(L\right)\subset I_{f}$, and so $\left(JR_{L}\right)^{-1}\widehat{T}F$ is a $\bigvee$-closed sublattice of $ C \left(L\right)$, which contains $\1_{L},u$. By construction we have $\left(JR_{L}\right)^{-1}\widehat{T}F= C \left(L\right)$, so that $JR_{L} C \left(L\right)\subset \widehat{T}F$. Since both $J$ and $R_{L}$ are injections, it follows that $\card F<\kappa\le \card C \left(L\right)=\card JR_{L} C \left(L\right)\le\card \widehat{T}F\le \card F$, contradiction.
\end{proof}


\section{Embeddings between free vector lattices}\label{embed}

Let $\mathbf{VL}$ be the category whose objects are vector lattices and morphisms are vector lattice homomorphisms and let $\mathbf{BL}$ be its subcategory which consists of Banach lattices and contractive (i.e. of norm at most $1$) vector lattice homomorphisms. Note that these categories are not subcategories of the categories considered in the preceding sections. Unlike in the aforementioned categories, there is a free vector lattice over any set $A$; it will be denoted $\fvl\left(A\right)$, or $\fvl\left(n\right)$ in the case when $A=\left\{1,...,n\right\}$. Note that $\fvl\left(A\right)$ is the vector sublattice of $\R^{\R^{A}}$ generated by the coordinate projections (see \cite[Theorem 2.3]{bleier}). There is also a notion of a free Banach lattice $\fbl\left(A\right)$, although its definition is slightly different from the one given in Section \ref{fre}. For general information about free vector and Banach lattices the reader is recommended to consult with \cite{dw}.\medskip

As mentioned before, $\fvl\left(n\right)$ can be identified with the vector sublattice of $\R^{\R^{n}}$ generated by the coordinate projections, i.e. the space of lattice-linear functions on $\R^{n}$. Every element $f\in\fvl\left(n\right)$ can be represented via $f=\bigvee\limits_{i=1}^{m}g_{i}-\bigvee\limits_{j=1}^{n}h_{j}$, where $g_{i}$'s and $h_{j}$'s are linear functions. Moreover, $f$ agrees with $g_{i_{0}}-h_{j_{0}}$ on the set $$\left\{x\in\R^{n},~\forall i=1,...,m:~ g_{i_{0}}\left(x\right)\ge g_{i}\left(x\right)\right\}\cap \left\{x\in\R^{n},~\forall j=1,...,n:~ h_{j_{0}}\left(x\right)\ge h_{j}\left(x\right)\right\},$$ which is a polyhedral cone (cf. \cite[Corollary 4.3.2]{keimel}).\medskip

Let $S_{n-1}$ be the sphere in $\R^{n}$ with respect to the $\|\cdot\|_{\8}$-norm, i.e. the sphere in $\ell_{n}^{\8}$. Since all elements of $\fvl\left(n\right)$ are positively homogeneous, the restriction operator $f\mapsto \left.f\right|_{S_{n-1}}$ is an isomorphic embedding of $\fvl\left(n\right)$ into $ C \left(S_{n-1}\right)$. The range of this embedding is a sublattice of $ C \left(S_{n-1}\right)$. It is easy to see that this sublattice contains $\1$ and separates points of $S_{n-1}$, and is therefore dense in $ C \left(S_{n-1}\right)$\footnote{This sublattice is in fact the space of all piecewise affine functions on $S_{n-1}$, but we will not need this, and could not find a reference for it.}.\medskip

Recall that any positive operator between order unit spaces is norm-continuous. Hence, any vector lattice homomorphism from $\fvl\left(n\right)$ into $\fvl\left(m\right)$ can be uniquely extended to a homomorphism from $ C \left(\So_{n-1}\right)$ into $ C \left(\So_{m-1}\right)$. Recall also, that any norm-dense sublattice of $ C \left(K\right)$ is order dense, and that a homomorphism is injective once its restriction to an order dense sublattice is injective. Therefore, every injective vector lattice homomorphism from $\fvl\left(n\right)$ into $\fvl\left(m\right)$ can be extended to an injective homomorphism from $ C \left(\So_{n-1}\right)$ into $ C \left(\So_{m-1}\right)$. This immediately leads to the following result.

\begin{proposition}
If $\fvl\left(m\right)\simeq \fvl\left(n\right)$, then $m=n$.
\end{proposition}
\begin{proof}
Let $T:\fvl\left(m\right)\to \fvl\left(n\right)$ be an isomorphism, so that $T^{-1}:\fvl\left(n\right)\to \fvl\left(m\right)$ is also an isomorphism. By the discussion above, both of them extend to injective homomorphisms $S: C \left(\So_{m-1}\right)\to C \left(\So_{n-1}\right)$ and $R: C \left(\So_{n-1}\right)\to C \left(\So_{m-1}\right)$. Moreover, it is easy to prove by continuity that $R=S^{-1}$, and so both $S,R$ are vector lattice isomorphisms. It follows that $\So_{m-1}$ and $\So_{n-1}$ are homeomorphic (this can be deduced from \cite[Theorem 2.34]{ab}), and so $m=n$.
\end{proof}

We can relatively easily generalize this result to the case of not necessarily finite sets using \cite[Lemma 3.9]{dw}. However, we will go further than that. On the way to the main result of the section we deal with its partial case.

\begin{proposition}\label{finite}
If $\fvl\left(m\right)$ embeds injectively into $\fvl\left(n\right)$, then $m\le n$.
\end{proposition}
\begin{proof}
Let $T:\fvl\left(m\right)\to \fvl\left(n\right)$ be an injective homomorphism and let $S: C \left(\So_{m-1}\right)\to C \left(\So_{n-1}\right)$ be its extension, as before. According to \cite[Theorem 2.34]{ab} there are $w:\So_{n-1}\to\R_{+}$ and $\varphi:\So_{n-1}\to \So_{m-1}$ such that $\left[Sf\right]\left(x\right)=w\left(x\right)f\left(\varphi\left(x\right)\right)$, for every $x\in \So_{n-1}$ and $f\in  C \left(\So_{m-1}\right)$. Moreover, injectivity of $S$ implies that $\varphi\left(U\right)$ is dense in $\So_{m-1}$, where $U:=\left\{x\in\So_{n-1},~ w\left(x\right)> 0\right\}$. Indeed, otherwise find a non-zero $f\in C \left(\So_{m-1}\right)$ which vanishes on $\varphi\left(U\right)$; then $Sf=\0$, contradicting injectivity. Also, note that $w=S\1=T\1\in \fvl\left(n\right)$.

Fix $i=1,...,m$. Let $p_{i}:\R^{m}\to \R$ be the projection onto the $i$-th coordinate. Note that $\left.p_{i}\right|_{\So_{m-1}}\in\fvl\left(m\right)$. It follows that $g_{i}:=w\cdot p_{i}\circ\varphi=S\left.p_{i}\right|_{\So_{m-1}}=T\left.p_{i}\right|_{\So_{m-1}}\in \fvl\left(n\right)$. Hence, for every $x\in U$ we have that $\varphi\left(x\right)=\frac{1}{w\left(x\right)}\left(g_{1}\left(x\right),...,g_{m}\left(x\right)\right)$. Since $\varphi\left(x\right)\in\So_{m-1}$, it follows that $\bigvee\limits_{i=1,...,m}\left|g_{i}\left(x\right)\right|=w\left(x\right)$, for every $x\in U$.

Next, we can extend $w,g_{1},...,g_{m}$ by positive homogeneity to get piecewise linear functions $\hat{w},\hat{g}_{1},...,\hat{g}_{m}:\R^{n}\to\R$ with pieces being convex polyhedral cones in $\R^{n}$, as discussed at the beginning of the section. Moreover, we have $\bigvee\limits_{i=1,...,m}\left|\hat{g}_{i}\left(x\right)\right|=\hat{w}\left(x\right)$, for every $x\in V$, where $V=\R_{+}U=\left\{x\in\R^{n},~ \hat{w}\left(x\right)> 0\right\}$. Furthermore, density of $\varphi\left(U\right)$ in $\So_{m-1}$ now becomes density of $\hat{\varphi}\left(V\right)$ in $\R^{m}$, where $\hat{\varphi}$ is the extension of $\varphi$ by positive homogeneity, so that $\hat{\varphi}\left(x\right)=\frac{\|x\|_{\8}}{\hat{w}\left(x\right)}\left(\hat{g}_{1}\left(x\right),...,\hat{g}_{m}\left(x\right)\right)$, for every $x\in V$.

We can now split $V$ into a finite number of convex polyhedral cones in $\R^{n}$ such that on each of these cones $W$ each of $\hat{w},\hat{g}_{1},...,\hat{g}_{m}$ agrees with a restriction of a linear function. Moreover, we can further split $W=\bigcup\limits_{i=1}^{m}W_{i}^{\pm}$ where $W_{i}^{+}:=\left\{x\in W,~ \hat{w}\left(x\right)=\hat{g}_{i}\left(x\right)\right\}$ and $W_{i}^{-}:=\left\{x\in W,~ \hat{w}\left(x\right)=-\hat{g}_{i}\left(x\right)\right\}$ and $W=\bigcup\limits_{j=1}^{n} {}_{j}^{\pm}W$, where ${}_{j}^{+}W:=\left\{x\in W,~ \|x\|_{\8}=x_{j}\right\}$ and ${}_{j}^{-}W:=\left\{x\in W,~ \|x\|_{\8}=-x_{j}\right\}$.

Now assume that $m>n$. To reach a contradiction it is enough to show that $\hat{\varphi}\left(\Int \left(W_{i}^{\pm}\cap {}_{j}^{\pm}W\right)\right)$ is nowhere dense in $\R^{m}$, for every $i=1,...,m$, $j=1,...,n$, and every choice of $\pm$. This leaves us with the following configuration:

\begin{itemize}
\item $Q$ is an open polyhedral cone in $\R^{n}$;
\item $f_{1},...,f_{m}$ -- linear functions on $\R^{n}$;
\item $0\ne f_{1}\left(x\right)\ge \pm f_{i}\left(x\right)$ and $0\ne x_{1}\ge \pm x_{j}$, for every $i=1,...,m$, $j=1,...,n$ and $x\in Q$.
\end{itemize}
The goal is to show that $\psi\left(Q\right)$ is nowhere dense in $\R^{m}$, where $\psi\left(x\right)=\frac{x_{1}}{f_{1}\left(x\right)}\left(f_{1}\left(x\right),...,f_{m}\left(x\right)\right)$, for every $x\in Q$.

Note that $Q$ is a semi-algebraic set (see \cite[Definition 2.1.1]{br}) and $\psi$ is rational, hence a semi-algebraic map (see \cite[Definition 2.3.2]{br} and \cite[Exercise 2.3.9(2,3,5)]{br}). Therefore, $\psi\left(Q\right)$ is also semi-algebraic (see \cite[Theorem 2.3.4]{br}), moreover $\dim \overline{\psi\left(Q\right)}=\dim \psi\left(Q\right)\le \dim Q=n$ (see \cite[Exercise 2.5.4(1,2)]{br} and \cite[Definitions 2.4.2 and 2.5.3]{br} for the concept of the dimension). Thus, $\overline{\psi\left(Q\right)}$ is a finite union of sets homeomorphic to $\R^{k}$, for $k\le n<m$ (see \cite[Section 2.4.5, comment (2)]{br}), and so cannot contain an open set in $\R^{m}$.
\end{proof}

\begin{remark}
It is desirable to find a more elementary way to complete the argument. Note that in the proof we obtained some information which we ended up not using, but perhaps it could be useful for an alternative approach. Furthermore, Proposition \ref{finite} is algebraic in nature, and so one could seek an entirely algebraic proof for it.
\qed\end{remark}

\begin{theorem}\label{mfre}
Let $A$ and $B$ be sets. Then, $\fvl\left(A\right)$ embeds injectively into $\fvl\left(B\right)$ if and only if $\card A\le \card B$.
\end{theorem}
\begin{proof}
Sufficiency has been addressed in \cite[Proposition 3.5(1)]{dw}, and so we only prove necessity. The case when both $A,B$ are finite is considered in Proposition \ref{finite}, and so we may assume that $A$ is infinite and that $\card A> \card B$. Let $J:\fvl\left(A\right)\to\fvl\left(B\right)$ be an injective homomorphism. Also, let $\iota_{A}:A\to \fvl\left(A\right)$ and $\iota_{B}:B\to \fvl\left(B\right)$ be as in the definition of a free object in Section \ref{fre}. Recall that $\fvl\left(B\right)=\vlt\left(\iota_{B}\left(B\right)\right)$ (see \cite[Proposition 3.2]{dw}). Hence, for every $a\in A$ there is a finite $B_{a}\subset B$ such that $J\iota_{A}\left(a\right)\in\vlt\left(\iota_{B}\left(B_{a}\right)\right)$.

Since the cardinality of the collection of the finite subsets of $B$ is smaller than $\card A$, there is a finite $B'\subset B$ and infinite $A'\subset A$ such that $B_{a}=B'$, for every $a\in A'$. Let $n:=\card B'$ and let $A''$ be a finite subset of $A'$ of cardinality $n+1$. We have $J\iota\left(A''\right)\subset\vlt\left(\iota_{B}\left(B'\right)\right)$, and since $J$ is a homomorphism, we get $J\vlt\left(\iota\left(A''\right)\right)\subset\vlt\left(\iota_{B}\left(B'\right)\right)$. Now, recall that $\vlt\left(\iota\left(A''\right)\right)\simeq \fvl\left(n+1\right)$ and $\vlt\left(\iota_{B}\left(B'\right)\right)\simeq \fvl\left(n\right)$ (see \cite[Proposition 3.5]{dw}). Thus, the restriction of $J$ is an injective homomorphism from a vector lattice isomorphic to $\fvl\left(n+1\right)$ into a vector lattice isomorphic to $\fvl\left(n\right)$, contradicting Proposition \ref{finite}.
\end{proof}

Note that the corresponding result does not hold for free Banach lattices.

\begin{proposition}
$\fbl\left(\omega\right)$ embeds injectively into $\fbl\left(2\right)$.
\end{proposition}
\begin{proof}
Recall that $\fbl\left(\omega\right)$ is a sublattice of $ C \left(K\right)$, where $K:=\left[-1,1\right]^{\omega}$ (see \cite[Corollary 5.7]{dw}), while $\fbl\left(2\right)\simeq C \left(\So_{1}\right)$ (this follows from \cite[Proposition 5.3]{dw}). Since $K$ is a metrizable connected locally connected compact space there is a surjective continuous map from $\left[-1,1\right]$ onto $K$ (see \cite[Exercise 6.3.14]{engelking}), and there is also a surjective continuous map from $\So_{1}$ onto $\left[-1,1\right]$. Hence, there is a surjective continuous map $\varphi:\So_{1}\to K$, and then $f\mapsto f\circ\varphi$ defines an injective homomorphism from $ C \left(K\right)$ into $ C \left(\So_{1}\right)$. Its restriction to $\fbl\left(\omega\right)$ is the required embedding.
\end{proof}

\begin{remark}
Similar arguments show that $\fbl\left(m\right)$ embeds as a closed sublattice of $\fbl\left(n\right)$, for any natural $m,n>1$. However, $\fbl\left(A\right)$ does not embed as a closed sublattice of $\fbl\left(n\right)$, for any infinite $A$ and $n\in\N$, since $\fbl\left(n\right)\simeq  C \left(\So_{n-1}\right)$, and so its closed sublattices are isomorphic to AM-spaces, which is impossible for $\fbl\left(A\right)$ (see \cite[Theorem 8.2]{dw}).
\qed\end{remark}

We say that $e_{1},...,e_{n}$ in a vector lattice $E$ are \emph{vector lattice independent} if no non-trivial lattice linear expression with $e_{1},...,e_{n}$ results in $0_{E}$. Recall that if $h$ is a lattice linear expression viewed as a function on $\R^{n}$ (and so an element of $\fvl\left(n\right)$), then $h\left(e_{1},...,e_{n}\right)=\widehat{\varphi}\left(h\right)$, where $\widehat{\varphi}:\fvl\left(n\right)\to E$ is generated by $\varphi: i\mapsto e_{i}$, $i=1,...,n$. Hence, $e_{1},...,e_{n}$ are vector lattice independent if and only if $\widehat{\varphi}$ is injective. It follows from Proposition \ref{finite} that there are no $n+1$ vector lattice independent elements in $\fvl\left(n\right)$. This idea can be generalized to infinite sets.\medskip

\begin{remark}\label{independent}
It is easy to see that if $e_{1},...,e_{n}$ are vector lattice independent, then removing any element results in a strictly smaller generated vector sublattice of $E$. The converse is false: in $\R$ the set $\left\{1\right\}$ is obviously a minimal generating set, but it is not vector lattice independent, since $1\vee 0 -1 =0$.
\qed\end{remark}

\section*{Acknowledgements}

The authors would like to thank Kevin Abela, Emmanuel Chetcuti, Anke Kalauch, Mitchell Taylor, Pedro Tradacete, Vladimir Troitsky, Onno van Gaans and Marten Wortel for helpful conversations on the topic of the paper. Additional credit goes to Janko Stennder and David Mu\~noz-Lahoz for informing the authors about \cite{kvgs} and \cite{keimel}, respectively. Part of the work on the article took place during the Workshop on Free Banach Lattices and the subsequent visit of the second author to the University of Wuppertal. The authors are grateful to Jochen Gl\"uck for organizing the conference, hospitality and some ideas related to the discussion preceding Remark \ref{independent}. The material of Section \ref{embed} was largely inspired by the discussion on the first day of the aforementioned conference sparked by some questions posed by Noa Bihlmaier. The model GPT-6 Sol was used for proofreading and to generate an idea for Remark \ref{exam}.

The first author was supported by ERDF and MICIU/AEI/ 10.13039/501100011033 project PID2021-122126NB-C32 and by Fundaci\'{o}n S\'{e}neca - ACyT Regi\'{o}n de Murcia project 21955/PI/22

The second author was supported by grant PCI2024-155094-2 funded by\linebreak MICIU/AEI/10.13039/501100011033

\end{document}